\documentclass[a4paper,11pt]{amsart}

\usepackage{amsmath,amsfonts,amssymb,graphics,amsthm,mathrsfs,pgfplots}
\usepackage{hyperref}
\usepackage{enumitem}
\usepackage{bbm}
\usepackage{mathtools}
\usepackage{a4wide}
\usepackage{dsfont}
\usepackage[foot]{amsaddr}
\usepackage{geometry}
\newtheorem{theorem}{Theorem}[section]
\newtheorem{lemma}[theorem]{Lemma}
\newtheorem{proposition}[theorem]{Proposition}

\newtheorem{definition}[theorem]{Definition}
\theoremstyle{remark}
\newtheorem{remark}[theorem]{Remark}
\newtheorem{remarks}[theorem]{Remarks}

\newtheorem{question}[theorem]{Question}

\newcommand{\id}{\mathds{1}}
\newcommand{\R}{\mathbb{R}}
\renewcommand{\P}{\mathbb{P}}
\newcommand{\E}{\mathbb{E}}
\newcommand{\Z}{\mathbb{Z}}

\newcommand{\eps}{\varepsilon}

\begin{document}

\title[The maximum of a strongly correlated Gaussian process]{The maximum of a strongly \\ correlated Gaussian process} 
\author{Jason Li}
\author{Stephen Muirhead}
\email{ylii0409@student.monash.edu,stephen.muirhead@monash.edu}
\address{School of Mathematics, Monash University}

\date{} 
\date{\today}

\begin{abstract}
We revisit a result of Mittal--Ylvisaker that states that the rescaled maximum of a stationary sequence of Gaussian random variables has a Gaussian limit if correlations decay sufficiently slowly. Taking a new approach we relax the conditions for the Gaussian limit and give an extension to smooth non-stationary random fields.
\end{abstract}

\maketitle


\section{Introduction}

A foundational result in extreme value theory is that the maximum of an i.i.d.\ collection of Gaussian random variables has an asymptotically Gumbel law after suitable rescaling \cite{ft28,gne43}. Berman \cite{ber64} showed that this remains true for stationary sequences of dependent Gaussian random variables $(X_n)_{n \in \Z}$ if the covariance $K(n) = \textrm{Cov}[X_0, X_n]$ satisfies $K(n) \log n \to 0$ as $n \to \infty$. If correlations decay more slowly, Mittal and Ylvisaker showed that a Gaussian limit can arise:

\begin{theorem}[\cite{my75, my76}]
\label{t:my}
Let $(X_n)_{n \in \Z}$ be a stationary Gaussian sequence with covariance $K(n) = \textrm{Cov}[X_0 ,X_n]$. Suppose that, as $n \to \infty$, $K(n) \to 0$ monotonically and $K(n) \log n \to \infty$ eventually monotonically. Then there exist $a_n, b_n$ such that
\[  b_n (M_n - a_n)   \stackrel{d}{\Longrightarrow} Z    \quad \text{as } n \to \infty , \]
where $M_n = \max_{0 \le i \le n} X_n$ and $Z$ is a standard Gaussian random variable.
\end{theorem}

Mittal-Ylvisaker also gave two extensions to this result: they studied the boundary regime $K(n) \log n \to \mu \in (0,\infty)$, in which the maximum has a mixed Gumbel-Gaussian limit, and proved analogous results for certain continuous non-smooth stationary Gaussian processes.

\vspace{0.1cm}
The monotonicity conditions in Theorem \ref{t:my} are fairly restrictive, for instance they imply that correlations are non-negative. However some regularity is necessary since, as also shown in \cite{my75}, a Gaussian limit does not arise in general under the sole assumption that $K(n) \log n \to \infty$. 

\vspace{0.1cm}
In this note we revisit the limit theorem of Mittal-Ylvisaker, making several new contributions:
\begin{itemize}
\item We relax the monotonicity assumptions, replacing them with an `asymptotic flatness' condition which can be viewed as a slight strengthening of the usual notion of `slow variation'.
 \item We extend the result in three directions: to smooth processes, to random fields in higher dimensions, and to certain non-stationary processes.
\item We provide a completely different proof, which exploits a link between the Gaussian limit and the closeness of the maximum to its `best linear approximation'.
\end{itemize}

\subsection{Main result}

Let $f$ be a continuous Gaussian process on either $\R^d$ or $\Z^d$, and let $K(x,y) = \textrm{Cov}[f(x) ,f(y)]$ be its covariance kernel. Throughout we assume that $f$ is centred.

\smallskip
We introduce the following notion of regularity:

\begin{definition}[Asymptotic flatness]
\label{d:af}
We say that $K$ is asymptotically flat if there exists a function $w : \R^+ \to \R$ such that $K(x,y) \sim w(|x-y|)$ as $|x-y| \to \infty$, and which satisfies the following:  for every $\eta > 0$ there exists a $\beta > 0$ such that
\begin{equation}
\label{e:af}
  \limsup_{v \to \infty}  \sup_{u \in [v^{-\beta} , 1] } \Big| \frac{ w( u v) }{ w(v) } - 1 \Big| \le \eta  .
\end{equation}
\end{definition}

We now state our main result:
\begin{theorem}
\label{t:main}
Suppose that $\textrm{Var}[f(x)] = 1$ for every $x$, $K(x,y) \log |x-y| \to \infty$ as $|x-y| \to \infty$, and $K$ is asymptotically flat in the sense of Definition \ref{d:af}. Then there exist $a_R, b_R$ such that
\[  b_R (M_R - a_R)   \stackrel{d}{\Longrightarrow} Z    \quad \text{as } R \to \infty ,\]
where $M_R = \sup_{x \in [0,R]^d} f(x)$ and $Z$ is a standard Gaussian random variable.
\end{theorem}

\begin{remarks}
$\,$
\begin{enumerate}[label=(\alph*)]
\item  Theorem \ref{t:main} contains Theorem \ref{t:my} as a special case. Indeed suppose, as in Theorem~\ref{t:my}, that $f$ is stationary and $K(n) \log n$ is non-decreasing. Then for every $\eta > 0$,
\[ K(n^{1/(1+\eta)})  \log(n^{1/(1+\eta)}) \le K(n) \log n \quad \Longrightarrow \quad K(n) \ge \frac{1}{1+\eta} K(n^{1/(1+\eta)}) .\]
Hence if $K(n)$ is also non-increasing, $K$ satisfies \eqref{e:af} (take $\beta = \eta/ (1+\eta) )$.
\vspace{0.1cm}

\item Ho--McCormick \cite[Corollary 2.4]{hm99} gave a variant of Theorem \ref{t:my} which relies on a different notion of regularity. Precisely they assumed, as well as $K(n) \log n \to \infty$, that
\begin{equation}
\label{e:hm}
  \frac{1}{n} \sum_{k=1}^n |K(k)-K(n)| = o(1/\log n) .
  \end{equation}
There is no strict relationship between \eqref{e:hm} and \eqref{e:af}: \eqref{e:hm} requires that $K$ be inside a narrow window of size $\ll 1/\log n$ \textit{on average}, whereas \eqref{e:af} requires that $K$ stay \textit{uniformly} within a much \textit{wider} window of size $\ll K(n)$. 

\item Recall that a function $w : \R^+ \to \R$ is said to be slowly varying (at infinity) if for every $u > 0$, $\lim_{v \to \infty} w(uv) / w(v) = 1$. In fact this implies that, for every $s \in (0,1)$,
\[   \sup_{u \in [s ,1] } \Big| \frac{ w( u v) }{ w(v) } - 1 \Big| \to 0 . \] 
The function $w$ in the definition of asymptotic flatness is necessarily slowly varying. 

\begin{question}
Does the conclusion of Theorem \ref{t:main}  hold if we replace `asymptotic flatness' with the assumption that $K(x,y) \sim w(|x-y|)$ for a slowly varying $w$?  
\end{question}

\item The proof in \cite{my75,my76} shows that, in Theorem \ref{t:my}, one can take $a_n = (1- K(n))^{1/2} \gamma_n$ and $b_n  = K(n)^{-1/2}$ for an explicit universal $\gamma_n$. Our approach shows that one can take $b_R = w(R)^{-1/2}$ in Theorem \ref{t:main}, but only provides an implicit expression $a_R = \E[M_R]$ for the centring constant.

\begin{question}
Under the conditions of Theorem \ref{t:main}, is 
\[ \E[M_R] =  (1- w(R))^{1/2} \gamma_R + o(w(R)^{-1/2}) \] 
for some explicit $\gamma_R$? What does $\gamma_R$ depend on?
\end{question}

\noindent Comparison with \cite{my75,my76} suggests that, in the discrete case, $\gamma_R$ is universal (i.e.\ depends only on the dimension), and in the continuous stationary case, depends only on the behaviour of $K$ at the origin.

\item We use a box $[0,R]^d$ in the definition of $M_R = \sup_{x \in [0,R]^d} f(x)$ for simplicity; in fact it can be replaced with any other rescaled set $RD$, $D$ a compact Lipschitz domain, without significant change to the proof.
\end{enumerate}

\end{remarks}

\subsection{Comparison of methods}
Let us describe the main idea of our proof and contrast it to that of \cite{my75, my76} and also of \cite{hm99}.

\smallskip
Our main idea is to show that the centred maximum $M_R - \E[M_R]$ is well-approximated by its `best linear approximation', i.e.\ the random variable  $Q_1(M_R)$ in the linear span of $f(x)$ which minimises $\E[(M_R - \E[M_R] -  Q_1(M_R))^2]$. Since $Q_1(M_R)$  is necessarily a Gaussian variable, the Gaussian limit for $M_R$ follows from showing that 
\[ \frac{ \textrm{Var}[Q_1(M_R)] }{\textrm{Var}[M_R] } \to 1 \quad \text{as } R \to \infty . \]
To achieve this we provide an explicit expression for $Q_1(M_R)$, which is new to our knowledge, and use it to lower bound the variance of $Q_1(M_R)$. By comparing this with a suitable upper bound on the variance of $M_R$, arising from the hypercontractivity approach of Chatterjee \cite{cha14}, we obtain the Gaussian limit.

\smallskip
In \cite{my75, my76}  the main idea is instead to decompose the process as $f  \stackrel{d}{=} b_n Z \oplus g$ where $Z$ is a standard Gaussian and $g$ is an independent Gaussian process. By using a comparison inequality of Berman \cite{ber64} one can compare the maximum of $b_n Z \oplus g$ with the maximum of  $b_n Z \oplus \hat{g}$, where $\hat{g}$ is an i.i.d.\ sequence of Gaussians. By showing that the maximum of the latter process has a limiting Gaussian distribution under the assumptions of Theorem \ref{t:my}, the Gaussian limit can be established for the original process. Significant technical difficulties arise in the use of Berman's comparison inequality, and to carry out this step appears to require strong regularity assumptions.   
 
 \smallskip
 The approach of \cite{hm99} is somewhat closer to ours; they consider the maximum of the process recentred by its partial average, and show that it is a Gaussian process with rapidly decaying correlations. Since the average is a Gaussian variable, and the maximum of the remainder process has smaller fluctuations, they are able to deduce the Gaussian limit. Their use of a linear statistic (the average) to approximate the maximum is reminiscent of our approach, however the relationship between the average and the best linear approximation $Q_1(M_R)$ is not clear in general.

\subsection{Acknowledgements}
This research was carried out as part of the first author's Honours Research Project at Monash University, supervised by the second author. The second author is supported by the Australian Research Council Future Fellowship FT240100396.

\medskip
\section{Best linear approximation for the maximum}

In this section we study the best linear approximation for the maximum. For simplicity we work with finite-dimensional Gaussian vectors; later we apply the results to continuous processes via approximation (see also Remark \ref{r:cp}).

\smallskip
Let $X = (X_i)_{1 \le i \le n}$ be a centred Gaussian vector with covariance $K(i,j) = \E[X_i X_j]$, let $M = \max_{1 \le i \le n} X_i$, and let $I \in \{1,\ldots,n\}$ be the index that attains the maximum, i.e.\ such that $X_I = M$. Under the assumption that $\textrm{Corr}(X_i, X_j) < 1$ for all $i \neq j$, $I$ is a.s.\ unique. By Gaussian concentration, $\textrm{Var}[M] \le \max_i \textrm{Var}[X_i]$ (see e.g.\ \cite[(1.1)]{cha14}), so that $M \in L^2$.

\smallskip
 Let $\mathcal{L}$ be the linear span of $X$, i.e.\ random variables $Y = \sum_{1 \le i \le n} \alpha_i X_i$ for some $\alpha_i \in \R$. Let $Q_1(M)$ denote the $L^2$-projection of $M - \E[M]$ onto $\mathcal{L}$. Equivalently, $Q_1(M)$ is the element of $\mathcal{L}$ that minimises $\E[ (M- \E[M] - Q_1(M) )^2 ]$. Observe that $Q_1(M)$ has a natural interpretation as the `best linear approximation' of $M$. It can also be viewed as the first component in the Weiner-It\^{o} chaos decomposition of $M$ (see \cite{jan97}).

\smallskip
The main result of this section is an explicit expression for $Q_1(M)$:

\begin{proposition}
\label{p:bla}
Suppose that $\textrm{Corr}(X_i, X_j) < 1$ for all $1 \le i \le j \le n$. Then
\[ Q_1(M) = \sum_{1 \le i \le n}   \mathbb{P}[I = i] X_i  . \]
In particular
\begin{equation}
\label{e:varq1}
 \textrm{Var} \big[Q_1(M) \big] = \sum_{1 \le i,j \le n} K(i,j)  \mathbb{P}[I = i]  \mathbb{P}[I = j] .
 \end{equation}
\end{proposition}

\begin{proof}
Without loss of generality we may assume that $X$ is non-degenerate (otherwise add an i.i.d.\ Gaussian vector of small variance and take limits). Let $F : \R^n \to \R$ be a smooth function with bounded first derivatives such that $\E[F(X)^2] < \infty$. Denote by $Q_1(F)$ the projection of $F(X) - \E[F(X)]$ onto $\mathcal{L}$. We first claim that
\begin{equation}
\label{e:smooth}
 Q_1(F) = \sum_{1 \le i \le n}  \E[ \partial_i F(X) ] X_i  . 
 \end{equation}
Indeed writing $X = Q Z$, for $Q = (Q_{i,j})$ and $(Z_i)$ a set of i.i.d.\ standard Gaussian random variables, and $G = F \circ Q$, then
\[ Q_1(F) = Q_1(G)  = \sum_{1 \le i \le n}  \E[ Z_i G(Z) ] Z_i  . \]
Integrating by parts 
\[ \E[ Z_i G(Z) ]  = \E[ \partial_i G(Z) ]  = \sum_{1 \le j \le n} Q_{i,j} \E[ \partial_j F(X) ]  \]
and so 
\[ Q_1(F) = \sum_{1 \le i \le n}  \E[ Z_i G(Z) ] Z_i   =  \sum_{1 \le j \le n}  \E[ \partial_j F(X) ]  \sum_{1 \le i \le n} Q_{i,j} Z_i = \sum_{1 \le j \le n}  \E[ \partial_j F(X) ] X_j . \]
We now apply \eqref{e:smooth} to the soft-max function
\[ F_\beta(X) := \beta^{-1} \log \Big( \sum_i e^{\beta X_i}  \Big)  \]
which satisfies $M \le  F_\beta(X)  \le M + \log n/\beta$. A direct computation gives
\[ \partial_i F_\beta(X) = \frac{e^{\beta X_i} }{\sum_i e^{\beta X_i} } \in [0,1] .\]
Amost surely, as $\beta \to \infty$, 
\[ \ \frac{e^{\beta X_i} }{\sum_i e^{\beta X_i} } \to  \id_{I = i}(X) , \]
hence by dominated convergence  $\lim_{\beta \to \infty} \E[  \partial_i F_\beta(X)  ]   =  \P[I = i] $. Inserting into \eqref{e:smooth} we see that
\begin{equation}
\label{e:conv}
\lim_{\beta \to \infty} Q_1(F_\beta) =   \sum_{1 \le i \le n} \P[I = i]  X_i 
\end{equation}
almost surely and in $L^2$.  Since $F_\beta(X) \to M$ in $L^2$, we also have $Q_1(F_\beta) \to Q_1(M)$ in $L^2$, so that we can identify $Q_1(M)$ with the right-hand side of \eqref{e:conv}.
\end{proof}

\begin{remark}
\label{r:cp}
Although we bypass this route, one could also prove a continuum analogue of Proposition \ref{p:bla} directly. Indeed, suppose that $f$ is a continuous centred Gaussian process on a compact set $ T \subset \R^d$, and let $M = \sup_{x\in T} f(x)$. Then, under suitable assumptions on $f$ and $T$, one can show that the best linear approximation of $M$ is $\int_{x \in T}  f(x) d\mu_I(x)$, where $d\mu_I$ is the distribution of $I$.
\end{remark}

We contrast \eqref{e:varq1} with a well-known `interpolation' formula for the variance of the maximum:
\begin{proposition}[{\cite[(1.4)]{cha14}}]
\label{p:var}
Suppose that $\textrm{Corr}(X_i, X_j) < 1$ for all $i \neq j$. For $t \in [0,1]$, let $X^t  \stackrel{d}{=}t X + \sqrt{1-t^2} \tilde{X}$, where $\tilde{X}$ is an independent copy of $X$, and define $I^t$ analogously to $I$ for the vector $X^t$ replacing $X$. Then
\[ \textrm{Var}[M] =  \int_0^1\sum_{1 \le i ,j \le n}  K(i,j)   \mathbb{P}[I = i, I^t = j]  \, dt  . \]
\end{proposition}
The process $X^t$ defines an interpolation between $X$ and an independent copy $\tilde{X}$ along the Ornstein-Uhlenbeck semigroup. Note that, compared to \cite[(1.4)]{cha14}, we prefer to use a different parameterisation for this interpolation (our $t$ is $e^{-t}$ in \cite[(1.4)]{cha14}).

\medskip
\section{Gaussian limit for the maximum}

In this section we complete the proof of Theorem \ref{t:main}. For brevity we focus on case of continuous processes, since the proof in the discrete case is simpler. Throughout the section we work under the assumptions of Theorem \ref{t:main}.

\smallskip
Define $\Lambda_R = [0,R]^d$, and let $I_R \in \Lambda_R$ be the index that attains the maximum, i.e.\ such that $M_R = f(I_R)$. Let us assume for the time being that $K(x,y) < 1$, $x \neq y$, so that $I_R$ is uniquely defined almost surely (\cite[Lemma 2.6]{kp90}); we lift this assumption at the end of the section.

\subsection{Discretisation}
It will be convenient to study a discretised version of the maximum defined as follows. Let $\eps > 0$ be discretisation parameter. Define $\Lambda_{R,\eps} = \Lambda_R \cap  (\eps \Z^d)$, and let $X_{R,\eps}$ be the finite-dimensional Gaussian vector $f |_{\Lambda_{R,\eps}}$. Let $M_{R,\eps} := \max_{x \in \Lambda_{R,\eps} } f(x) \le M_R$ be the maximum of $X_{R,\eps}$, and define $I_{R,\eps} \in \Lambda_{R,\eps}  \subset \Lambda_R$, the (a.s.\ unique) index of the maximum.

\smallskip
Let us record here the key properties of the discretisation:

\begin{lemma}
\label{l:dis}
For every $R > 0$, as $\eps \to 0$
\[ M_{R,\eps} \to M_R \qquad \text{and} \qquad I_{R,\eps} \to I_R \]
almost surely and in $L^2$.
\end{lemma}
\begin{proof}
Since $f$ is continuous, and its maximum $M_R$ is attained at the a.s.\ unique point $I_R$, the almost sure convergence of $M_{R,\eps} \to M_R$ and $I_{R,\eps} \to I_R$ is immediate. The $L^2$ convergence follows by dominated convergence since $M_{R,\eps} \le M_R \in L^2$ and $\| I_{R,\eps}\|_{\infty} \le R$.
\end{proof}

Since $M_{R,\eps} \to M_R$ in $L^2$, the limit of the projections $\lim_{\eps \to 0} Q_1(M_{R,\eps})$ exists in $L^2$, and we denote this limit by $Q_1(M_R)$. One can view $Q_1(M_R)$ as the projection of $M_R$ onto the linear subspace spanned by $f$, but for our purposes it is enough to know that $Q_1(M_R)$ is a Gaussian variable (as the $L^2$ limit of Gaussian variables).

\subsection{Hypercontractivity and delocalisation} 
We next state some auxiliary estimates that involve the hypercontractivity property of the Ornstein-Uhlenbeck semigroup, applied to the discretised vector $X_{R,\eps}$. In the following lemma we abbreviate $I = I_{R,\eps}$:

\begin{lemma}[Hypercontractivity]
\label{l:hyp} 
For $t \in [0,1]$, let $f^t  \stackrel{d}{=} t f + \sqrt{1-t^2} \tilde{f}$, where $\tilde{f}$ is an independent copy of $f$, and define $I^t$ analogously to $I$ for $f^t$ replacing $f$. Then for every $S \subset \Lambda_{R,\eps}$,
\[ \P[I \in S, I^t \in S] \le  \P[I \in S]^{1+(1-t)/2}.\]
\end{lemma}
\begin{remark}
The bound is clearly not tight for $t \approx 0$ since in fact $ \P[I \in S, I^0 \in S] =  \P[I \in S]^{2}$. The interesting feature is that $\P[I \in S, I^t \in S] \ll \P[I \in S]$ near the degeneracy point $t \approx 1$.
\end{remark}
\begin{proof}
We follow the proof of similar results in \cite{cha14}. Let $q, q' > 0$ be such that $1/q + 1/q' = 1$. Then by H\"{o}lder's inequality
\begin{equation}
\label{e:hyp}
 \P[ I \in S, I^t \in S] = \E \big[ \id_{I \in S} \P [ I^{t} \in S  | f ]  \big] \le  \P[I \in S]^{1/q} \Big( \E \big[ \P[ I^t \in S | f ]^{q'} \big] \Big)^{1/q'} . 
 \end{equation}
The hypercontractivity property of the Ornstein-Uhlenbeck semigroup (see \cite[Appendix B]{cha14}, but note the different parameterisation used therein) states that for any $p > 1$, $F \in L^p$, and $t \in (0,1]$, 
\[ \Big( \E\big[  \E \big[ F[f^t] \big| f \big]^{q' } \big] \Big)^{1/q' }  \le  \Big( \E[F[f]^{p} ] \Big)^{1/p} \]
where $q'= 1 + (p-1)/t^2 > 1$. Setting $p = 1 + t$, applying this to $F = \id_{I \in S}$, and inserting in \eqref{e:hyp}, gives
\[ \P[ I \in S, I^t \in S] \le  \P[ I \in S]^{1/q + 1/p} = \P[I \in S]^{1 - 1/q' + 1/p} = \P[I \in S]^{\alpha(t)} \]
where 
\[ \alpha(t) = 1 + \frac{1}{1+t}- \frac{1}{1+1/t} .\] 
One can check that
\[ \frac{1}{1+t} - \frac{1}{1+1/t} \ge \frac{1-t}{2} \ , \quad t \in [0,1], \]
(the left-hand side is convex and its derivative at $t=1$ is $-1/2$) which gives the result.
\end{proof}

To apply the above estimate we need to know that $I_R$ is not too localised. We prove this using a basic estimate on the growth of the maximum:

\begin{lemma}[Growth of the maximum]
\label{l:max}
Let $M_R^x = \max_{y \in x + \Lambda_R} f(y)$, so that $M_R = M_R^0$. Then
\[ \frac{M^x_R}{\sqrt{2d \log R}}  \stackrel{\P}{\longrightarrow} 1 \ , \quad \text{as } R \to \infty , \]
uniformly over $x \in \mathbb{R}^d$.
\end{lemma}
\begin{proof}
If $f$ is a stationary Gaussian process on $\Z$ or $\R$ such that $K(x) \to 0$ as $x \to 0$, this was proven in \cite[Theorem 3.4]{pic67}. It is easy to check that the proof goes through also for non-stationary processes under the assumptions that $K(x,x) = 1$ for all $x$ and $K(x,y) \to 0$ as $|x-y| \to \infty$, and that the convergence in probability is uniform if $K(x,y) \le \omega(|x-y|)$ for some $\omega(s) \to 0$ as $s \to \infty$.
\end{proof}

We deduce the follow delocalisation estimate for $I_R$:

\begin{lemma}[Delocalisation of the maximum]
\label{l:deloc}
For every $\beta > 0$ there exists a $\beta' > 0$ such that, for sufficiently large $R$
 \[  \sup_{x \in \Lambda_{R}}   \P \big[  I_R \in x + \Lambda_{R^{1-\beta}}  \big]  \le R^{-\beta'} . \]
\end{lemma}
\begin{proof}
It suffices to prove the result for $\beta \in (0,1)$. Let $\bar{M}_R^x = \max_{y \in x + \Lambda_{R^{1-\beta}} } f(y)$. Observe that for every $t > 0$, 
\begin{equation}
\label{e:union}
\{ I_R \in x + \Lambda_{R^{1-\beta}} \}  \Longrightarrow \{ \bar{M}_R^x \ge t \} \cup \{M_R \le t\}  .
\end{equation}
Fix $t_0 \in (\sqrt{2d (1-\beta)} , \sqrt{2d} ) $. Combining Lemma \ref{l:max} with the fact that, by Gaussian concentration,
\[  \textrm{Var} \Big[ \sup_{x \in T} f(x) \Big] \le \max_y \textrm{Var}[f(y)] = 1 \quad  \forall T \subset \R^d \]
we have, as $R \to \infty$,
\[  \sup_{x \in \R^d} \Big| \frac{ \E[\bar{M}^x_R]  }{  \sqrt{2d (1- \beta) \log R}  } - 1 \Big| \to 0 \quad \text{and} \quad   \E[ M_R ] \sim \sqrt{2d \log R} .\]
Taking $t = t_0 \sqrt{\log R} $ in \eqref{e:union} and applying the union bound and the Borell-TIS inequality (\cite[Appendix A5]{cha14}) yields that for $R \to \infty$ eventually
\begin{align*}
 \P[ I_R \in  x + \Lambda_{R^{1-\beta}} ] & \le  \P[ \bar{M}_R^x \ge t_0 \log R] + \P[ M_R \le t_0 \log R] \\
 &  \le  e^{- (t_0 \log R -\E[  \bar{M}_R^x ] )^2 / 2} +  e^{- (\E[  M_R ]  - t_0 \log R )^2 / 2}  \\
 &  \le e^{-(t_0 -  \sqrt{2d \beta} )^2  (\log R)  /3} + e^{-(\sqrt{2d} - t_0)^2 (\log R) / 3}  
 \end{align*}
uniformly over $x \in \Lambda_R$, which gives the result.
\end{proof}

\subsection{Variance bounds}
 Next we prove lower and upper bounds on $ \textrm{Var}[Q_1(M_R)] $ and $ \textrm{Var}[M_R] $ respectively, from which the Gaussian limit for $M_R$ will follow. Let $w$ be the function provided by the assumption of asymptotic flatness.
 
\begin{proposition}
\label{p:lb}
\[ \liminf_{R \to \infty} \liminf_{\eps \to 0}  w(R)^{-1} \textrm{Var}[ Q_1( M_{R,\eps}  )  ] \ge 1. \]
\end{proposition}

\begin{proposition}
\label{p:ub}
\[ \limsup_{R \to \infty} \limsup_{\eps \to 0}  w(R)^{-1}  \textrm{Var}[ M_{R,\eps}    ] \le 1 . \]
\end{proposition}

\begin{proof}[Proof of Proposition \ref{p:lb}]
Fix $\eta > 0$, and let $\beta > 0$ be as in Definition \ref{d:af}. Let $\mathcal{S}$ be the subset of $(x,y) \in \Lambda_{R,\eps}^2$ such that $\|x-y\|_\infty \le R^{1-\beta}$. Applying \eqref{e:varq1} to the vector $X_{R,\eps}$,
\begin{align*}
 &\textrm{Var}[ Q_1( M_{R,\eps}  )  ]   =  \sum_{x,y \in \Lambda_{R,\eps}} K(x,y) \P[I_{R,\eps}=x] \P[I_{R,\eps}=y] \\
&   \qquad    = \sum_{(x,y) \in \Lambda_{R,\eps}^2 \setminus \mathcal{S} } K(x,y) \P[I_{R,\eps}=x] \P[I_{R,\eps}=y]  + \sum_{(x,y) \in \mathcal{S}} K(x,y) \P[I_{R,\eps}=x] \P[I_{R,\eps}=y] \\
&  \qquad   =: A_{R,\eps} + B_{R,\eps}.
 \end{align*}
By the asymptotic flatness of $w$, for sufficiently large $R$ we can bound 
 \begin{align*}
 A_{R,\eps} & \ge   \min_{(x,y) \in \Lambda_{R,\eps}^2 \setminus \mathcal{S} }  K(x,y) \sum_{(x,y) \in \Lambda_{R,\eps}^2 \setminus \mathcal{S} } \P[I_{R,\eps}=x] \P[I_{R,\eps}=y]  \\
 &  \ge w(R)(1-\eta)  \sum_{(x,y) \in \Lambda_{R,\eps}^2 \setminus \mathcal{S} } \P[I_{R,\eps}=x] \P[I_{R,\eps}=y]  \\
 & \ge  w(R)(1-\eta) \Big( 1 - \sum_{(x,y) \in \mathcal{S}}\P[I_{R,\eps}=x] \P[I_{R,\eps}=y]  \Big)  .
 \end{align*} 
 Moreover we have
 \[ \sum_{(x,y) \in \mathcal{S}}\P[I_{R,\eps}=x] \P[I_{R,\eps}=y]  \le \sup_{x \in \Lambda_{R,\eps}}   \P \big[ I_{R,\eps} \in  x +   \Lambda_{R^{1-\beta}} \big] .  \]
 Since $I_{R,\eps} \to I$ in probability as $\eps \to 0$, we have   
 \begin{equation}
 \label{e:limi}
  \limsup_{\eps \to 0} \sup_{x \in \Lambda_{R,\eps}}   \P \big[ I_{R,\eps}  \in x + \Lambda_{R^{1-\beta}} \big]  \le \sup_{x \in \Lambda_{R}}   \P[  I_R \in x + \Lambda_{2   R^{1-\beta} } ]  . 
  \end{equation}
By Lemma \ref{l:deloc}, the right-hand side of \eqref{e:limi} is at most $R^{-\beta'} $ for some $\beta'>0$ and sufficiently large $R$. Combining we have
 \begin{equation}
 \label{e:a}
 \liminf_{R \to \infty} \liminf_{\eps \to 0}  w(R)^{-1} A_{R,\eps} \ge 1-\eta .
 \end{equation}
 Since $|K(x,y)| \le  \sqrt{ K(x,x) K(y,y) } = 1$, we similarly have
  \begin{align*}
   \limsup_{\eps \to 0} |B_{R,\eps}|   & \le  \limsup_{\eps \to 0} \sum_{(x,y) \in \mathcal{S}}\P[I_{R,\eps}=x] \P[I_{R,\eps}=y]   \le  \sup_{x \in \Lambda_{R,\eps}}   \P[ I_R \in   x + \Lambda_{ 2 R^{1-\beta} }  ]    \le R^{-\beta'}  .
   \end{align*}
   Since $w(R)^{-1} R^{-\beta} \to 0$ by assumption, we conclude that 
    \[ \lim_{R \to \infty} \limsup_{\eps \to 0}  w(R)^{-1} |B_{R,\eps}| = 0 .\] 
The result follows by combining with \eqref{e:a} and taking $\eta \to 0$.
\end{proof}

\begin{proof}[Proof of Proposition \ref{p:ub}]
Fix $\eta , \beta > 0$ and $\mathcal{S}$ as in the previous proof.  Abbreviate $I = I_{R,\eps}$, and define  $f^t$ and $I^t$ as in the statement of Lemma \ref{l:hyp}. Applying Proposition \ref{p:var} to the vector $X_{R,\eps}$,
\begin{align*}
 &\textrm{Var}[  M_{R,\eps}  ]   =  \int_0^1 \sum_{x,y \in \Lambda_{R,\eps}} K(x,y) \P[I =x, I^t=y] dt  \\
&   \qquad    = \int_0^1 \sum_{(x,y) \in \Lambda_{R,\eps}^2 \setminus \mathcal{S} } K(x,y) \P[I=x,I^t =y] \,dt  + \int_0^1 \sum_{(x,y) \in \mathcal{S}} K(x,y) \P[I =x, I^t =y] \, dt  \\
&  \qquad   =: A_{R,\eps} + B_{R,\eps}.
 \end{align*}
 
 By the asymptotic flatness of $w$, for sufficiently large $R$
 \begin{align}
 \label{e:a2}
 \nonumber & w(R)^{-1} A_{R,\eps}   \le  w(R)^{-1} \max_{(x,y) \in \Lambda_{R,\eps}^2 \setminus \mathcal{S} }  \int_0^1  \P[I =x, I^t=y]  \, dt    \\
\nonumber &  \qquad   \le (1 + \eta)  \int_0^1  \sum_{(x,y) \in \Lambda_{R,\eps}^2 } \P[I =x, I^t=y]   \, dt   \\ 
&  \qquad =  1 + \eta .
 \end{align}
 
Let $(\mathcal{C}_i)_{\in \Z^d}$ denote the set of boxes $\Lambda_{R^{1-\beta}} + i R^{1-\beta}$. For each $i \in \Z^d$ let $\widetilde{\mathcal{C}}_i = \mathcal{C}_i \cup \bigcup_{j \sim i}\mathcal{C}_i$, where $j \sim i$ denotes the $3^d - 1$ sites in $\Z^d$ at unit distance to $i$ in the sup-norm. Since $|K(x,y)| \le 1$, and applying the hypercontractivity estimate Lemma \ref{l:hyp},
\begin{align*}
 & B_{R,\eps}   \le  \int_0^1  \sum_i  \P[ I  \in \widetilde{\mathcal{C}}_i , I^t \in \widetilde{\mathcal{C}}_i ] \, dt    \\
&  \qquad   \le   \int_0^1  \sum_i  \P[ I  \in \widetilde{\mathcal{C}}_i ]^{1 + (1-t)/2} \, dt  \\
&   \qquad   \le  \Big( \sum_i  \P[ I  \in \widetilde{\mathcal{C}}_i ] \Big) \int_0^1  \max_i  \P[ I  \in \widetilde{\mathcal{C}}_i ]^{(1-t)/2}  \, dt \\
& \qquad \le 3^d \int_0^1  \max_i  \P[ I  \in \widetilde{\mathcal{C}}_i ]^{(1-t)/2}  ,
 \end{align*}
 where the final inequality is since each $x \in \Lambda_{R,\eps}$ appears in at most $3^d$ distinct $\widetilde{\mathcal{C}}_i$. Since $\int_0^1 \alpha^{(1-t)/2} \, dt = 2/ |\log p|$ for $p \in (0,1)$, we see that
  \[ B_{R,\eps} \le \frac{ 2 \cdot 3^d}{ | \log \max_i  \P[ I  \in \widetilde{\mathcal{C}}_i ] | } .\]
  Since $I := I_{R,\eps} \to I_R$ in probability as $\eps \to 0$, we have 
 \begin{equation}
 \label{e:limi2}
 \limsup_{\eps \to 0}   \max_i  \P[ I  \in \widetilde{\mathcal{C}}_i ]  \le \sup_{x \in \Lambda_R} \P[I_R \in x + \Lambda_{4 R^{1-\beta}} ] . 
 \end{equation}
   By Lemma \ref{l:deloc}, the right-hand side of \eqref{e:limi2} is at most $R^{-\beta'} $ for some $\beta'>0$ and sufficiently large $R$, so that 
 \[ \limsup_{\eps \to 0 } w(R)^{-1} B_{R,\eps}  \le  \frac{ 2 \cdot 3^d}{ \beta' } \frac{1}{ w(R) \log R  }   \]
 for large $R$. Since $w(R) /\log R  \to \infty$ by assumption, the result follows by combining with \eqref{e:a2} and taking $\eta \to 0$.
\end{proof}

\subsection{Completion of the proof}
Theorem \ref{t:main} is an easy consequence of Propositions \ref{p:lb} and \ref{p:ub}. The following elementary claim will be convenient. For a non-degenerate random variable $Y \in L^2$, define its standardisation $\hat{Y} = (Y - \E[Y]) / \sqrt{\textrm{Var}[Y] } $. 

\vspace{0.1cm}
\textbf{Claim:} Let $X,Y \in L^2$ be non-degenerate random variables such that $\textrm{Cov}[X-Y,Y] = 0$. Then 
\begin{equation}
\label{e:varbound}
\E \big[ \big(\hat{X} - \hat{Y} \big)^2 \big] =  2 \Big( 1 - \frac{ \sqrt{\textrm{Var}[Y] }}{\sqrt{\textrm{Var}{X} }} \Big) .
 \end{equation}
 
 \begin{proof}
Use that $\E[\hat{X} \hat{Y} ]    =    \frac{\textrm{Cov}[X,Y]   }{ \sqrt{ \textrm{Var}[X] \textrm{Var}[Y] } }   =  \frac{\textrm{Cov}[X-Y, Y] + \textrm{Var}[Y]   }{ \sqrt{ \textrm{Var}[X] \textrm{Var}[Y] } } = \frac{ \sqrt{\textrm{Var}[Y] }}{\sqrt{\textrm{Var}{X} }}$.
 \end{proof}

\begin{proof}[Proof of Theorem \ref{t:main}]
By construction, $M_{R,\eps} - Q_1( M_{R,\eps} )  $ is orthogonal to $ Q_1( M_{R,\eps} )$, so that in particular  $ \textrm{Var}[ Q_1( M_{R,\eps} ) ]  \le  \textrm{Var}[ M_{R,\eps} ]$. Combining Propositions \ref{p:lb} and \ref{p:ub} we have that
\begin{equation}
\label{e:varasy}
\lim_{R \to \infty}  \limsup_{\eps \to 0} \Big| \frac{  \textrm{Var}[ Q_1( M_{R,\eps} ) ] }{  \textrm{Var}[ M_{R,\eps} ]  } - 1 \Big|  = 0
\end{equation}
and also that 
\begin{equation}
\label{e:varasy2}
\lim_{R \to \infty}  \limsup_{\eps \to 0} \Big| \frac{ \textrm{Var}[ M_{R,\eps} ] }{ w(R) } - 1 \Big| = 0.
\end{equation}
Applying \eqref{e:varbound} to $X = M_{R,\eps}$ and $Y =  Q_1(M_{R,\eps})$, \eqref{e:varasy} gives that
\begin{equation}
\label{e:lim1}
\lim_{R \to \infty} \limsup_{\eps \to 0}  \E \big[ \big(\hat{M}_{R,\eps}  - \hat{Q}_1( M_{R,\eps}  ) \big)^2 \big] = 0 . 
\end{equation}
By Lemma \ref{l:dis},  $\lim_{\eps \to 0} M_{R,\eps} = M_R$ and $\lim_{\eps \to 0} Q_1( M_{R,\eps}  ) = Q_1(M_R) $ in $L^2$, so that \eqref{e:varasy2}--\eqref{e:lim1} yields
\begin{equation}
\label{e:lim2}
\lim_{R \to \infty}  \E \big[ \big(  \hat{M}_R  - \hat{Q}_1( M_{R}  ) \big)^2 \big] = 0 
\end{equation}
and
\begin{equation} 
\label{e:lim3}
 \lim_{R \to \infty}    \frac{ \textrm{Var}[ M_R ] }{ w(R) }  = 1 
 \end{equation}
 Since $\hat{Q}_1( M_R ) \stackrel{d}{=} Z$, a standard Gaussian random variable, \eqref{e:lim2} implies that $\hat{M}_R$ converges in distribution to $Z$. Setting $a_R = \E[ M_R ]$ and $b_R = \sqrt{ \textrm{Var}[M_R]}$, this is the conclusion of Theorem \ref{t:main}. By \eqref{e:lim3} we may replace $b_R$ with $\sqrt{w(R)}$ without change to the conclusion.
 
 This completes the proof under the assumption that $K(x,y) < 1$ for every $x \neq y$. Let us briefly explain how to modify the proof if this assumption fails. Let $\kappa(x) > 0$ be a function to be determined later. Let $g$ be a smooth centred stationary Gaussian field on $\mathbb{R}^d$ satisfying $E[g(0)^2] = 1$ and $\E[g(0) g(x)] = 0$ for $|x| > 1$ (it is easy to see such a field exists). Since $g$ necessarily satisfies $\textrm{Corr}(g(x),g(y)) < 1$, we may apply the previous arguments to the field 
 \[ f'(x) := \frac{ f(x) +  \kappa(x) g(x) }{ 1 + \kappa^2(x)} \]
  to conclude that, letting $M'_R$ be the analogue of $M_R$ for $f'$,
  \[ w(R)^{-1/2} (M'_R - a_R) \stackrel{d}{\Longrightarrow} Z.\]
Thus it suffices to show that
 \[ |M_R - M_R'|  \le o_\mathbb{P}(w(R)^{1/2}) . \]
 Letting $I'_R$ be the analogue of $I_R$ for $f'$, by Lemma \ref{l:deloc} we have $\mathbb{P}(I'_R \in \Lambda_R \setminus \Lambda_{\sqrt{R}}) \to 1$. Since also
  \begin{align*}
 w(R)^{-1/2} |M_R - M_R'| \id_{I'_R \in \Lambda_R \setminus \Lambda_{\sqrt{R}}} & \le w(R)^{-1/2} \sup_{x \in   \Lambda_R \setminus \Lambda_{\sqrt{R}}} \Big| f(x) - \frac{f(x) + \kappa(x) g(x)}{1+ \kappa^2(x)}  \Big| \\
  & \le  w(R)^{-1/2}  \sup_{x \in   \Lambda_R \setminus \Lambda_{\sqrt{R}}} \kappa^2(x) |f(x) | + \kappa(x) |g(x)| ,
  \end{align*}
in turn it suffices that 
\[ w(R)^{-1/2} \sup_{x \in   \Lambda_R \setminus \Lambda_{\sqrt{R}}} \kappa^2(x) |f(x) | + \kappa(x) |g(x)|   \stackrel{\P}{\to} 0 . \]
This can be guaranteed by choosing $\kappa(x)$ to decay sufficiently rapidly.
\end{proof}

\medskip
\bibliographystyle{halpha-abbrv}
\bibliography{max}

@article{ber64,
	title = {Limit theorems for the maximum term in stationary sequences},
	author = {S. M. Berman},
	journal = {Ann. Math. Stat.},
	volume = {35},
	number = {2},
	year = {1964},
	pages = {502--516},
	}

@Book{cha14,
 Author = {Chatterjee, S.},
 Title = {Superconcentration and related topics},
 Journal = {Springer Monographs in Mathematics},
 Year = {2014},
 Publisher = {Springer},
 Language = {English},
}

@article{ft28,
	author={R.A. Fishcer and L.H. Tippett},
	title={Limiting forms of the frequency distribution of the largest member of a sample},
	journal={Proc. Cambridge Phil. Sot.},
	volume={24},
	year={1928},
	pages={180--190},
	}

@article{gne43,
	author={B.V. Gnedenko},
	title={Sur la distribution limite du terme maximum d’une s\'{e}rie al\'{e}atoire},
	journal={Ann. Math.},
	volume={44},
	year={1943},
	pages={423--453},
	}

@article{hm99,
	author={H.-C. Ho and W.P. McCormick},
	title={Asymptotic distribution of sum and maximum for {G}aussian processes},
	journal={J. Appl. Prob.},
	volume={36},
	pages={1031--1044},
	year={1999}
}

@Book{jan97,
 Author = {Janson, S.},
 Title = {{Gaussian Hilbert spaces}},
 Journal = {{Camb. Tracts Math.}},
 Volume = {129},
 Year = {1997},
 Publisher = {Cambridge: Cambridge University Press},
}

@article{kp90,
	author={J. Kim and D. Pollard},
	title={Cube root asymptotics},
	journal={Ann. Stat.},
	volume={18},
	number={1},
	year={90},
	pages={191--219},
}

@article{my75,
	title = {Limit distributions for the maxima of stationary {G}aussian processes},
	author = {Y. Mittal and D. Ylvisaker},
	journal = {Stoc. Proc. Appl.},
	volume = {3},
	pages = {1--18},
	year = {1975},
	}

@article{my76,
	title = {Strong laws for the maxima of stationary {G}aussian processes},
	author = {Y. Mittal and D. Ylvisaker},
	journal = {Ann. Probab.},
	volume = {4},
	pages = {357--371},
	year = {1976},
	}

@article{pic67,
	title={Maxima of stationary {G}aussian processes},
	author={J. Picklands},
	journal={Probab. Theory. Related Fields},
	year={1967},
	volume={7},
	pages={190--223},
	}

\end{document}